\documentclass[12pt]{amsart}
\usepackage{xfrac} 
\usepackage{fullpage}
\usepackage{amssymb, comment, color}
\usepackage[toc,page]{appendix}
\usepackage{epsf,overpic,amssymb,amscd,times,euscript}
\usepackage{faktor}
\usepackage[utf8]{inputenc}
\usepackage[english]{babel}
\usepackage{amsmath}
\usepackage{amsthm}
\usepackage{amsfonts}
\usepackage{amssymb}
\usepackage{amscd}
\usepackage{bezier}
\usepackage{enumerate}
\usepackage{siunitx}
\usepackage{url}

\usepackage{breakurl}
\usepackage[breaklinks]{hyperref}

\usepackage{graphicx}

\usepackage{upgreek}
\numberwithin{equation}{section}

\usepackage[vmargin=1in,hmargin=1in]{geometry}
\usepackage{enumitem}
\setlist[description]{style=nextline}
\setlist[itemize,enumerate]{leftmargin=*,itemsep=0pt,parsep=0pt}
\emergencystretch=\maxdimen
\usepackage{parskip}

\usepackage{bbm}

\newtheorem{mainthm}{\sc Theorem}           
\newtheorem{thm}{Theorem}[section]               
\newtheorem*{thm*}{Theorem}               
\newtheorem*{cor*}{Corollary}        
\newtheorem{lem}[thm]{Lemma}  
\newtheorem*{lem*}{Lemma}
\newtheorem{prop}[thm]{Proposition}     
\theoremstyle{definition}
\newtheorem{rem}[thm]{Remark}           
   
 \newtheorem*{acknowledgement*}{\protect\acknowledgementname}
\newcounter{claim}
\newenvironment{claim}[1][]{\refstepcounter{claim}\par\noindent\underline{Claim~\theclaim:}\space#1}{}
\newenvironment{claimproof}[1]{\par\noindent\underline{Proof:}\space#1}{\qed}

 \providecommand{\acknowledgementname}{Acknowledgement}

\author{Abror Pirnapasov}

\address{Fakult\"at f\"ur Mathematik \\
Ruhr-Universit\"at Bochum}
\email{\texttt{Abror.Pirnapasov@rub.de }}

\title[Hutchings' inequality for the Calabi invariant revisited with an
application to pseudo-rotations.]{Hutchings' inequality for the Calabi invariant revisited with an
application to pseudo-rotations.}
\begin{document}

\begin{abstract}
In \cite{hutchings2016mean}, Hutchings uses embedded contact homology to show the following for area-preserving disc diffeomorphisms that are a rotation near the boundary of the disc: If the asymptotic mean action on the boundary is greater than the Calabi invariant, then the infimum of the mean action of the periodic points is less than or equal to the Calabi invariant.   In this article, we extend this to all area-preserving disc diffeomorphisms. Our strategy is to extend the diffeomorphism to a larger disc with nice properties and apply Hutchings' theorem.  
 As an application, we show that  the Calabi invariant of any smooth pseudo-rotation is equal to its rotation number.   
\end{abstract}

\maketitle
\section{Introduction}

Let $\phi$ be an area-preserving diffeomorphism of the closed 2 dimensional  radius $1$ disc $\mathbb{D}$ with the area form $\omega=r dr\wedge d\theta.$ Let $\beta=\frac{r^{2}}{2}d\theta$ be a primitive of $\omega.$ A  $C^{\infty}$ function $f:\mathbb{D}\rightarrow \mathbb{R} $ is called an action of $\phi$ if $f$ satisfies the following property:
     \begin{align}\label{action}
     \phi^{*}\beta-\beta=df.
     \end{align}
Such    $f:\mathbb{D}\to \mathbb{R}$ exists, because the 1-form $\phi^{*}\beta-\beta$ is closed, i.e., $d(\phi^{*}\beta-\beta)=\phi^{*}d\beta-d\beta=0$ and any closed 1-form on $\mathbb{D}$ is exact. Obviously $f$
is not unique, we address this in a moment.
The quantity $A_{\phi}(x)$ is called the \textit{asymptotic mean action} with respect to $f,$ if the following limit exists
$$A_{\phi}^{\infty}(x):=\lim_{n\rightarrow \infty}\frac{\sum_{i=0}^{n-1}f({\phi}^{i}(x))}{n}.$$
By Birkhoffs' ergodic theorem, $A_{\phi}^{\infty}(x)$ is well-defined for Lebesgue a.e. $x\in\mathbb{D}.$ If we choose $x\in\partial \mathbb{D},$ then  
$A_{\phi}^{\infty}$ is well-defined and is independent of $x.$ 

If $f$ is the action of $\phi$ with respect to the primitive $\beta,$ then for any $c\in \mathbb{R}$, $f+c$ is also and the corresponding asymptotic mean action simply changes by adding $c.$ For this reason we can define for each $a\in\mathbb{R}$ $$f_{(\phi,a)}:\mathbb{D}\to \mathbb{R}$$
to be the unique action function for $\phi,$ for which $A_{\phi}^{\infty}$ coincides with $a$ on $\partial \mathbb{D}.$ We define the corresponding Calabi invariant  by 
$$\mathcal{V}(\phi,a):=\frac{1}{\int_{\mathbb{D}}\omega}\int_{\mathbb{D}}f_{(\phi,a)\omega}.$$
We say $x$ is a periodic point of $\phi$ with period $d,$  if $\phi^{d}(x)=x$ for some $d\in \mathbb{Z}_{+}$ and  $\phi^{i}(x)\neq x$ for every $0\leq i\leq d-1.$ We will use the symbol $P(\phi)$ to denote the set of periodic points of $\phi.$  If  $x$ is a periodic point with period $d$, we define \textit{mean action}
$$A_{(\phi,a)}(x):=\frac{1}{d}\sum_{i=0}^{d-1}f_{(\phi,a)}({\phi}^{i}(x)).$$
\begin{rem} \label{remark}\begin{itemize}
    \item []
    \item We denote the ''usual'' asymptotic mean action and the Calabi invariant for maps $\phi$ with compact support by 
     \[
                           A_{(\phi)}^{\infty}(x):= A_{(\phi,0)}^{\infty}(x)
    \]
    and
    \[
                            \mathcal{V}(\phi):= \mathcal{V}(\phi,0).
    \]
    For the case where $\phi$ coincides with rotation by $2\pi\alpha$ on the boundary of $\mathbb{D},$ Hutchings in  \cite{hutchings2016mean} defines the asymptotic mean action and the Calabi invariant by $A_{(\phi)}^{\infty}(x):=A_{(\phi,\alpha)}^{\infty}(x)$ and $\mathcal{V}(\phi):=\mathcal{V}(\phi,\alpha).$
    \item Clearly $A_{\phi}^{\infty}(x)$ is well-defined at every $x\in P(\phi)$ and $A_{(\phi,a)}^{\infty}(x)=A_{(\phi,a)}(x).$ \item From the definitions of the Calabi invariant and the asymptotic mean action, we have $\mathcal{V}(\phi,a)=\mathcal{V}(\phi,0)+a$ and $A_{(\phi,a)}^{\infty}(x)=A_{(\phi,0)}^{\infty}(x)+a.$
    \item The Calabi invariant and the asymptotic mean action are  independent of the primitive of
$\omega$ and do not change under area-preserving conjugacy.  We refer for instance to \cite{mcduff2017introduction} or  \cite{Senior2020AsymptoticAA}.
\end{itemize}
\end{rem}

In \cite{hutchings2016mean}, Hutchings used  embedded contact homology to prove  the following inequality:
\begin{thm}[\cite{hutchings2016mean}]\ \label{hut}
 Let $\theta_{0}\in \mathbb{R},$ and let $\phi$ be an area-preserving diffeomorphism of $\mathbb{D}$
which agrees with a rotation by angle $2\pi\theta_{0}$ {near} the boundary. Suppose that
\begin{align}\label{1.2}
\mathcal{V}(\phi, \theta_{0}) < \theta_{0}. 
\end{align}
Then
\begin{align}\label{1.3}
\inf \bigg\{  A_{(\phi,\theta_{0})}(x)| x \in P(\phi)\bigg\}\leq \mathcal{V}(\phi, \theta_{0}).
\end{align}
\end{thm}
One of the conditions of Theorem \ref{hut} is that $\phi$ is a rigid rotation \textbf{near} the boundary i.e. on a neighborhood. In his blog post \cite{Hutpage}, Hutchings raises the question whether Theorem \ref{hut} holds when $\phi$ is a rigid rotation only \textbf{on} the boundary. The following result confirms this and moreover we show that inequality (\ref{1.3}) holds without any assumptions on the boundary except the inequality (\ref{1.2}). In this level of generality it is not immediately obvious what should play the role of $\theta_{0}$ in (\ref{1.2}) for a general area preserving disc map $\phi.$ Apriori this could involve the pointwise action or the mean action on the boundary. Our result shows it is enough to consider the mean action.

\begin{mainthm}\label{sub}
Suppose that $\phi$ is an area-preserving diffeomorphism of $\mathbb{D}$ satisfying 
\begin{align}\label{1.4}
\mathcal{V}(\phi, 0) < 0. 
\end{align}
Then
\begin{align}\label{1.5}
\inf\bigg\{ A_{(\phi,0)}(x)| x \in P(\phi)\bigg\}\leq \mathcal{V}(\phi, 0).
\end{align}
\end{mainthm} 

By Remark \ref{remark}, $$\mathcal{V}(\phi, \theta_{0})=\mathcal{V}(\phi, 0)+\theta_{0}$$ and  $$\inf \bigg\{  A_{(\phi,\theta_{0})}(x)| x \in P(\phi)\bigg\}= \inf \bigg\{  A_{(\phi,0)}(x)| x \in P(\phi)\bigg\}+\theta_{0},$$
so that (\ref{1.4}) and (\ref{1.5}) are nothing else than reformulations of (\ref{1.2}) and (\ref{1.3}) respectively. We found this reformulation is helpful because the mean action $A_{(\phi,0)}(x)$ and the Calabi invariant $\mathcal{V}(\phi, 0)$  are homogeneous. That is
\begin{align}\label{homoasym}
    A_{(\phi^{n},0)}(x)=nA_{(\phi,0)}(x)
\end{align}
and 
\begin{align}\label{homocal}
    \mathcal{V}(\phi^{n},0)=n\mathcal{V}(\phi, 0)
\end{align}
for every $n\in \mathbb{N}.$

The following theorem implies Theorem \ref{sub}.
\begin{thm}\label{ps3}
Suppose that $\phi$ is an area-preserving diffeomorphism of $\mathbb{D}$ satisfying
\begin{align}\label{1.6}
\mathcal{V}(\phi, 0) < -1000\pi. 
\end{align}
Then
\begin{align}\label{1.7}
\inf\bigg\{ A_{(\phi,0)}(x)| x \in P(\phi)\bigg\}\leq \mathcal{V}(\phi, 0).
\end{align} 
\end{thm}
Now we prove Theorem \ref{sub} using Theorem \ref{ps3}.
\begin{proof}[Proof of Theorem \ref{sub}] 
By assumptionon $\mathcal{V}(\phi,0)<0$, there exist natural number $n_{0}$ such that $$\mathcal{V}(\phi^{n_{0}},0)=n_{0}\mathcal{V}(\phi,0)< -1000\pi.$$  
By applying Theorem \ref{ps3} to the diffeomorphism $\phi^{n_{0}}$, we get 
$$\inf\bigg\{ A_{(\phi^{n_{0}},0)}(x)| x \in P(\phi)\bigg\}\leq \mathcal{V}(\phi^{n_{0}}, 0).$$
By (\ref{homoasym}) and (\ref{homocal}), we have 
$$n_{0}\inf\bigg\{ A_{(\phi,0)}(x)| x \in P(\phi)\bigg\}=\inf\bigg\{ A_{(\phi^{n_{0}},0)}(x)| x \in P(\phi)\bigg\}\leq \mathcal{V}(\phi^{n_{0}}, 0)=n_{0}\mathcal{V}(\phi, 0).$$\end{proof}

\subsection*{Applications to Pseudo-rotations}
In this section we apply our results to compute the Calabi invariant of pseudo-rotations.   In this article by a pseudo-rotation we always mean a $C^\infty$-smooth and area-preserving diffeomorphism of the closed disc $\mathbb{D}$ having a unique periodic point.  As is well known by results of Franks \cite{franks1988generalizations} the unique periodic point is an interior fixed point and to each pseudo-rotation $\phi$ there is a unique irrational number we denote by $\rho(\phi)\in[0,1)$ called the rotation number of $\phi$ which 
is characterised by the following dynamical property:  Every trajectory besides the unique fixed point has 
a well defined asymptotic mean winding number about the fixed point, and the value of this winding number is $\rho(\phi)$.  In particular $\rho(\phi)$ coincides with the rotation number of $\phi$ on the boundary of the disc.  A large class of pseudo-rotations that are not conjugate to a rotation were discovered by Anosov and Katok in their seminal paper \cite{katok1972new}.     

Applying the inequality (\ref{1.5}) above to pseudo-rotations yields the following identity:  

\begin{thm}[Calabi identity for pseudo-rotations]\label{ps2}
If $\phi:\mathbb{D}\to\mathbb{D}$ is a pseudo-rotation then its  Calabi invariant coincides with its rotation number.  In our notation 
\begin{equation}\label{CalabiID}
    \mathcal{V}{(\phi)}=\rho(\phi)
\end{equation}
or equivalently $\mathcal{V}{(\phi, 0)}=0$.  \end{thm}

In other words, the Calabi invariant of a pseudo-rotation coincides with the Calabi invariant of the rigid rotation with the same rotation number.  

In the proof we will also need the following: 

\begin{prop}\label{bb}\cite{bramham2015periodic}
If $\phi:\mathbb{D}\to\mathbb{D}$ is a pseudo-rotation, then the action of its unique fixed point $x_0$ is $\rho(\phi)$.   In other words $$A_{(\phi)}(x_{0})=A_{(\phi,\rho(\phi))}(x_{0})=\rho(\phi).$$\end{prop}

It is also possible to prove Proposition \ref{bb} with more elementary methods using a generalised version of Theorem 1.1 in \cite{Senior2020AsymptoticAA}.

\begin{proof}[Proof of Theorem \ref{ps2}]
The proof goes by contradiction. If the claim is false then either  $\mathcal{V}{(\phi, 0)}<0$ or $\mathcal{V}{(\phi, 0)}>0.$  

Note that the conclusion of Proposition \ref{bb} is equivalent to $A_{(\phi,0)}(x_{0})=0$ in our notation.   
If  $\mathcal{V}{(\phi, 0)}<0$, then by Theorem \ref{sub}, we have 
$$\inf\bigg\{ A_{(\phi,0)}(x)| x \in P(\phi)\bigg\}\leq \mathcal{V}(\phi, 0)<0=A_{(\phi,0)}(x_{0}).$$ Which implies the existence of another periodic point, and that is 
 contradiction to the uniqueness of periodic points.
 
If $\mathcal{V}{(\phi, 0)}>0:$ Then $\mathcal{V}{(\phi^{-1}, 0)}<0$ and we apply Thereom \ref{sub} to the map $\phi^{-1}.$\end{proof}

\begin{rem}
In \cite {senior2020relation}, the authors proved a version of Theorem \ref{sub} for Anosov Reeb flows using the Action-Linking Lemma. In \cite{cristofarogardiner2021contact}, authors recently show the same results as Theorem \ref{ps2} for all Reeb flow with only two periodic orbits on closed three-manifolds. Independently in \cite{beno}, the author proved Theorem \ref{ps2} by another method for pseudo-rotations with super-Liouvillean rotation number.
\end{rem}

\begin{rem}
On the proof of Theorem 1, we attempted to extend the diffeomorphism $\phi$ using a Hamiltonian, but we encountered the difficulty of bounding the mean action on the extended region.
\end{rem}

\subsection{Acknowledgement}

This work is part of the author’s Ph.D. thesis, written under the supervision of Barney Bramham and Gerhard Knieper at the Ruhr University of Bochum. The author is grateful  to Alberto Abbondandolo for his
suggestion of using generating functions and to  David Bechara Senior for the many helpful suggestions. The author was supported by the DFG SFB/TRR 191 ‘Symplectic Structures in Geometry, Algebra
and Dynamics’, Projektnummer 281071066-TRR 191.

\section{Generating functions for area-preserving maps of the strip.}
We first extend $\phi$ continuously to a larger disc, then approximate it by smooth maps in such a way that we can apply Hutchings' theorem. For the approximations, we use generating functions. Throughout, smooth mean $C^{\infty}$-smooth, unless otherwise stated.

We first state detailed facts about the generating functions given in \cite{abbondandolo2018sharp}.

Let $T:[a,b]\times \mathbb{R}\to [a,b]\times \mathbb{R}$
denote the map $T(r,\theta)=(r,\theta+2\pi)$ and set $\Omega=rdr\wedge d\theta$ where $(r,\theta)\in [a,b]\times \mathbb{R}\subset (0,+\infty)\times \mathbb{R}.$

Assume that $\Phi(r,\theta)=(R(r,\theta),\Theta(r,\theta))\colon [a,b]\times\mathbb{R}\to [a,b]\times\mathbb{R} $ is a  $C^{1}$ map with the following properties:
\begin{enumerate}
    \item $\Phi\circ T=T\circ \Phi,$
    \item $\Phi$ maps each connected component of $\partial( [a,b]\times\mathbb{R})$ into itself,
    \item $\Phi^{*}\Omega=\Omega,$
    \item $D_{1}R(r,\theta)>0.$
\end{enumerate}
Property (4) implies that $r\mapsto R(r,\theta)$ is an orientation-preserving diffeomorphism of $[a,b]$ onto itself, for every $\theta.$ This defines a smooth diffeomorphism
$$\Psi:[a,b]\times\mathbb{R}\to[a,b]\times\mathbb{R} \ \ \mbox{with} \ \ \Psi(r,\theta)=(R(r,\theta),\theta).$$
From the fact above, one can work with $(R,\theta)$ coordinates on $[a,b]\times \mathbb{R}.$ By property (2),
\begin{equation}\label{eq}
    R(b,\theta)=b
\end{equation}
for all $\theta\in\mathbb{R}.$

Let us consider the 1-form:
\begin{align}\label{for}
  \lambda(R,\theta)=\frac{R^2-r^2(R,\theta)}{2}d\theta+R(\theta-\Theta)dR.  
\end{align}
One can show easily that $T^{*}\lambda=\lambda$ with (\ref{for}).

By property (3) of $\Phi,$ we compute the differential of $\lambda$ : 
$$d\lambda=RdR\wedge d\theta-rdr\wedge d\theta-RdR\wedge d\theta+R dR\wedge d\Theta=$$
$$\Phi^{*}\Omega-\Omega=0.$$
The last equality shows that the 1-form $\lambda$ is closed and therefore exact on $[a,b]\times \mathbb{R}.$  Hence, we can find $W$ such that $dW=\lambda.$ From the equality $T^{*}\lambda=\lambda,$ it follows that $W\circ T$ is also solution of $dW=\lambda,$ hence $W\circ T=W+c$ for some constant $c\in\mathbb{R}.$ By (\ref{eq}),
$$\lambda_{(a,\theta)}[\partial_{\theta}]=0$$ for all $\theta\in\mathbb{R}$ and so $W$ is constant on the boundary of $[a,b]\times \mathbb{R}.$ Therefore $c=0,$ i.e., $W\circ T=W.$
We will call $W$  a \textit{generating function} for $\Psi$ with respect to the area form $\Omega.$ It is defined up to the addition of a constant and from the components of $dW=\lambda$ we see that $W$ satisfies:
\begin{align}\label{2.1}
    D_{1}W(R,\theta)=R(\theta-\Theta),
    \end{align}
    and
    \begin{align}\label{2.2}
    D_{2}W(R,\theta)=\frac{R^2-r^2}{2}.
\end{align}
\begin{rem}\label{nbhd}
We emphasize that if the property (4) is not true for the diffeomorphism $\Phi$ on the whole of $[a,b]\times \mathbb{R}$ but only true near $\{b\}\times \mathbb{R},$ then one can also define the generating function $W$ near to $\{b\}\times \mathbb{R}.$
\end{rem}
In section 2.6 of \cite{abbondandolo2018sharp}, the authors computed the action of a diffeomorphism with respect to the generating function of this diffeomorphism.
\begin{thm}\label{ac}The function $\Sigma:[a,b]\times\mathbb{R}\to \mathbb{R}$ given by  \begin{equation}\label{con}
   \Sigma:=W(R,\theta)+\frac{R^{2}\Theta-R^{2}\theta}{2}
   \end{equation}
   is an action of $\Phi(r,\theta)$ with respect to the one form  $\beta=\frac{r^{2}}{2}d\theta.$  That is, $\Sigma$ satisfies the equation (\ref{action}).
Moreover, $$\Sigma\circ T=\Sigma.$$\end{thm}
Note that $\Sigma$ and $W$ are unique up addition of a constant.

We now prove a technical result that plays an important role in proving Theorem \ref{ps3}.
\begin{lem} \label{lem} Let $W:[\frac{1}{2},1]\times \mathbb{R}\rightarrow \mathbb{R} $ and $V:[1,1+\beta]\times \mathbb{R}\rightarrow \mathbb{R}$ be smooth functions with the following properties:
\begin{itemize}
    \item $W\circ T=W$ and  $V\circ T=V,$
    \item $
   \hat{W}(r,\theta)= 
\begin{cases}
\  V(r,\theta),& \text{if } 1\leq r\leq 1+\beta \\
    W(r,\theta),              &  \text{if } r\leq 1
\end{cases}$ is a $C^{1}$ function, 
\item $D_{12}\hat{W}$ is continuous.
\end{itemize}
Then given any $\delta>0$ there exists a smooth function $Y:[\frac{1}{2},1+\beta]\times \mathbb{R}\to \mathbb{R}$ with the following properties:
\begin{enumerate}
    \item $Y\circ T=Y,$
    \item $Y\big|_{r\leq 1}=W,$
    \item $||\hat{W}-Y||_{C^{1}}+||D_{12}(\hat{W})-D_{12}(Y)||_{C^{0}}<\delta,$
    \item $Y\big|_{1+\frac{\beta}{2}\leq r\leq 1+\beta}=V.$
\end{enumerate}
\end{lem}
The proof of Lemma \ref{lem} is based on the following  classical approximation theorem \cite{whitney1934analytic} due to Whitney. 

\begin{thm}\label{Whit}
Let $M\subset\mathbb{R}^{2}$ be a
compact smooth submanifold with boundary.  Suppose that  $F\colon M \to  \mathbb{R}$ is a $C^{1}$ function  and that $D_{12}F$ exists and is continuous. 
Given any $\delta>0,$ there exists a $C^{\infty}$  smooth
function $\hat{F}:M\to\mathbb{R}$ such that $$||\hat{F}-F||_{C^{1}}+||D_{12}\hat{F}-D_{12}F||_{C^{0}}\leq \delta.$$ If $F$ is $C^{\infty}$ smooth  on a closed subset  $ \partial M\subset A\subset M,$
then $\hat{F}$ can be chosen to be equal to $F$ on $A.$
\end{thm}
\begin{rem}
We say $F:M\to R$ is smooth on $A\subset M$ if it has a smooth extension in a neighborhood of
each point of $A$.
\end{rem}
For the continuous version of Theorem \ref{Whit} see Theorem 6.21 in \cite{lee2013smooth}.

Let us denote by $\mathbb{D}(r)\subset\mathbb{R}^{2}$ the closed disc of radius $r>0$  which shares center with $\mathbb{D}$ and $A_{[r_{1},r_{2}]}=D(r_{2})\setminus \mathring{D}(r_{1}).$ Define a smooth map $\rho$ by 
$$\rho\colon [0,2]\times\mathbb{R}\to \mathbb{D}(2),$$
$$\rho(r,\theta)=re^{i\theta}$$
which restricts to a smooth covering map from $(0,2]\times \mathbb{R}$ to  the punctured disc $\mathbb{D}(2)\setminus\{0\}.$
\begin{proof}[Proof of Lemma \ref{lem}]  Because $\hat{W}\circ T=\hat{W}$ then there exists a function $W_{p}:A_{[\frac{1}{2},1+\beta]}\to \mathbb{R},$ such that  $W_{p}\circ \rho=\hat{W}.$  By assumptions on $\hat{W},$ the function
$W_{\rho}$ is $C^{1}$ and $D_{12}W_{\rho}$ is continuous on $A_{[\frac{1}{2},{1+\beta}]}.$ Additionally $W_{\rho}$ is  $C^{\infty}$ on $$A=A_{[\frac{1}{2},1]}\cup A_{[1+\frac{\beta}{2},1+\beta]}=\{(r,\theta)\in\mathbb{D}(1+\beta): \frac{1}{2}\leq r\leq 1 \ \mbox{and} \  1+\frac{\beta}{2}\leq r\leq 1+\beta\}.$$
Now, we apply Theorem \ref{Whit} to: the function $W_{\rho},$ and  as closed set $A$ as above and  $\delta>0.$
Then we get a smooth function $\hat{Y}$ with the following properties:
\begin{enumerate}
    \item $\hat{Y}\big|_{A}=W_{\rho},$
    \item $||W_{\rho}-\hat{Y}||_{C^{1}}+||D_{12}(W_{\rho})-D_{12}(\hat{Y})||_{C^{0}}<\delta.$
\end{enumerate}
By constructions, $Y=\hat{Y}\circ \rho$ fulfills all the properties we need.\end{proof}
\section{Proof of Theorem \ref{ps3}}
We now  prove the Theorem \ref{ps3} using  the previous lemmas.
\begin{proof}[Proof of Theorem \ref{ps3}]
 Assume 
$$\phi\big|_{\partial \mathbb{D}}=(1,g(\theta)).$$  By the area-preserving property of $\phi$, we know that $Dg: TS^{1}\to TS^{1}$ is positive. Because if $Dg$ negative, then the map $g:S^{1}\to S^{1}$ is the orientation-reversing diffeomorphism, which contradicts the orientation-preserving property of $\phi.$ By compactness of $S^{1},$ there exists a  positive $\epsilon$ such that 
\begin{align}\label{bound}
 \frac{1}{\epsilon}>||Dg||>\epsilon.
\end{align}
We now start to apply the properties given above to our situation. It is clear that one can find a smooth $\omega$-preserving diffeomorphism $\psi:\mathbb{D}\to \mathbb{D}$ which coincides with $\phi$ near the boundary of $\mathbb{D}$ and so that every $x\in \mathbb{D}(\delta)$ is a fixed point of $\psi$ for some sufficiently small $\delta>0.$ For example since $\phi$ can be  generated by a Hamiltonian we can multiply this by a cut off function vanishing on $\mathbb{D}(\delta).$

One can therefore find an $\Omega$-preserving smooth diffeomorphism $$\Upsilon(r,\theta)=(R(r,\theta),\Theta(r,\theta))\colon [0,1]\times\mathbb{R} \to [0,1]\times\mathbb{R}$$ with the following properties:
\begin{enumerate}
\item $\rho\circ\Upsilon=\psi\circ \rho$
\item $\Theta\big|_{r=1}=G(\theta)$ such that $|\theta-G(\theta)|\leq 2\pi$ for every $\theta\in \mathbb{R}.$ 
\end{enumerate}
In other words, $\Upsilon$ is the lift of $\psi$ except the origin of the disc $\mathbb{D}$ satisfying property (2).

By the left hand inequality (\ref{bound}), $D_{1}R(1,\theta)> \epsilon$ for all $
\theta\in\mathbb{R},$ since $\Upsilon$ is area-preserving and $D_{2}R(1,\theta)=0,$ and therefore
\begin{align}\label{pro1}
   D_{1}R(r,\theta)>0 
\end{align} near $\{1\}\times \mathbb{R}$ by continuity.
 
By Remark \ref{nbhd} and property (\ref{pro1}), we can find a generating function $W:(1-\kappa,1]\times \mathbb{R}\to \mathbb{R}$ of $\Upsilon$ near $\{1\}\times\mathbb{R}$ and satisfying the following equalities:
\begin{equation}\label{gen4}
    D_{1}W\big|_{r=1}=R(\theta-\Theta)\big|_{r=1}=(\theta-\Theta)\big|_{r=1}=\theta-G(\theta)
\end{equation}
and
 \begin{equation}\label{gen3}
     D_{2}W\big|_{r=1}=\frac{R^{2}-r^{2}}{2}\big|_{r=1}=0
 \end{equation}
on $\{1\}\times\mathbb{R}.$ 

By equation (\ref{gen3}), we know $W$ is constant on $\{1\}\times\mathbb{R}.$ So we assume 
\begin{equation}\label{onboun1}  
W\big|_{\{1\}\times\mathbb{R}}=0.
\end{equation}
For any $\beta,\delta>0,$ we can find a smooth function $f_{(\delta,\beta)}:[1,1+\beta]\to \mathbb{R}$ with the following properties:
\begin{enumerate}
    \item $f_{(\delta,\beta)}(1)=f_{(\delta,\beta)}(1+\beta)=0,$
    \item $f'_{(\delta,\beta)}([1,1+\beta])\subset (-\delta,1],$
    \item $f'_{(\delta,\beta)}(1)=1,$ 
    \item $f'_{(\delta,\beta)}(y)=0$  for every $y$  near $1+\beta.$
\end{enumerate}
We choose a smooth function $V_{\beta}(r,\theta):[1,1+\beta]\times \mathbb{R}\to \mathbb{R}$ in the following form $$V_{\beta}(r,\theta)=f_{(\epsilon,\beta)}(r)\cdot(\theta-G(\theta)).$$
\begin{claim}\label{bound1} The following inequalities:
\begin{enumerate}
    \item $|V_{\beta}|\leq 2\beta\pi,$
    \item $|D_{1}V_{\beta}|\leq 2\pi,$
    \item $D_{12}V_{\beta}-1\leq -\epsilon$
\end{enumerate}
hold.
\end{claim}
\begin{claimproof} The first and second inequalities follow easily from properties (1) and (2) of $f_{(\epsilon, \beta)}.$ Now, we prove the third inequality of the claim.

By estimate (\ref{bound}), we have the following inequality:
$$1-\epsilon>1-G'(\theta)>-\frac{1-\epsilon}{\epsilon}.$$
By property (2) of the  function $f_{(\epsilon,\beta)}(r),$  we have two cases $1\geq f_{(\epsilon,\beta)}'(r)\geq0$  and $~{0>f_{(\epsilon,\beta)}'(r)>-\epsilon.}$ 

\textbf{Case 1} If $1\geq f_{(\epsilon,\beta)}'(r)\geq 0,$ then 
$$D_{12}V-1=f_{(\epsilon,\beta)}'(r)(1-G'(\theta))-1\leq f_{(\epsilon,\beta)}'(r)(1-\epsilon)-1\leq 1-\epsilon-1=-\epsilon.$$

\textbf{Case 2} If $0>f_{(\epsilon,\beta)}'(r)>-\epsilon,$ then $$D_{12}V-1=f_{(\epsilon,\beta)}'(r)(1-G'(\theta))-1\leq -f_{(\epsilon,\beta)}'(r)\frac{1-\epsilon}{\epsilon}-1\leq 1-\epsilon-1=-\epsilon.$$\end{claimproof}

Now consider the following extension: $\hat{W}_{\beta}(R,\theta):(1-\kappa,1+\beta]\times \mathbb{R}\to \mathbb{R}$ given by
\[
   \hat{W}_{\beta}(R,\theta)= 
\begin{cases}
  V_{\beta}, & \text{if } 1\leq R\leq 1+\beta, \\
    W(R,\theta),              &  \text{if } R\leq 1.
\end{cases}
\]
Using equality (\ref{onboun1}) together with equations (\ref{gen3}) and (\ref{gen4}), one can easily show that  $\hat{W}_{\beta}(R,\theta)$ is a $C^{1}$ function and $D_{12}\hat{W}_{\beta}(R,\theta)$ is continuous. 

From now on, we assume $\beta \in (0,1).$

By Lemma \ref{lem} for $\delta=\min \big\{\epsilon,\frac{6(1-\beta)\pi}{3+\beta}\big\},$ there exists a smooth function $Y_{\beta}(r,\theta)$ such that 
\begin{enumerate}
    \item $d_{C^{1}}(\hat{W}_{\beta},Y_{\beta})+d_{C^{0}}(D_{12}(\hat{W}_{\beta})-D_{12}(Y_{\beta}))<\delta,$
    \item $Y_{\beta}\big|_{r\leq 1}=W,$
    \item $Y_{\beta}\big|_{r\geq 1+\frac{\beta}{2}}=V_{\beta},$
    \item $Y_{\beta}\circ T=Y_{\beta}.$
  \end{enumerate}
We now show that equations \ref{2.1}  and \ref{2.2},where $W=Y_{\beta}$, have a unique solution.

By Claim \ref{bound1}, we get 
$$D_{12}Y_{\beta}({R},\theta)-R\leq |D_{12}(Y_{\beta}({R},\theta)-V_{\beta}({R},\theta))|+D_{12}V_{\beta}({R},\theta)-{R}\leq $$
$$|D_{12}(Y_{\beta}({R},\theta)-V_{\beta}({R},\theta))|+D_{12}V_{\beta}({R},\theta)-1<\epsilon-\epsilon=0.$$
Therefore, by the Implicit function theorem, equations \ref{2.1}  and \ref{2.2}  have a unique solution where $W=Y_{\beta}$ and such solution defines a smooth area-preserving diffeomorphism $$\Phi_{\beta}(r,\theta)=(R_{\beta},{\Theta}_{\beta}):[1,1+\beta]\times \mathbb{R} \to [1,1+\beta]\times \mathbb{R}$$ with the following properties:
\begin{enumerate}
    \item $\Phi_{\beta}\big|_{r=1}=(1,g(\theta)),$
    \item $\Phi_{\beta}(x)=x$ for every $x$  near $\{1+\beta\}\times \mathbb{R},$
    \item $\Phi_{\beta}\circ T=T\circ \Phi_{\beta}.$
\end{enumerate}

From property (3) of $\Phi_{\beta},$ there exists a diffeomorphism $\hat{\phi}_{\beta} :A_{[1,1+\beta]}\to A_{[1,1+\beta]}$ such that  $$\hat{\phi}_{\beta}\circ\rho=\rho\circ{\Phi}_{\beta}.$$

The properties of $Y_{\beta}$ imply
\[
   \psi_{\beta}= 
\begin{cases}
   \hat{\phi}_{\beta},& \text{if } 1\leq r\leq 1+\beta, \\
   \psi,              &  \text{if } r\leq 1
\end{cases}
\]
is smooth. By construction of $\psi$, the map \[
   \phi_{\beta}= 
\begin{cases}
    \hat{\phi}_{\beta},& \text{if } 1\leq r\leq 1+\beta, \\
   \phi,              &  \text{if } r\leq 1
\end{cases}
\] is also smooth.

One can show the action of the diffeomorphism $\phi_{\beta}\colon\mathbb{D}(1+\beta)\to \mathbb{D}(1+\beta)$ is
\begin{equation}\label{imp}
  f_{({\phi}_{\beta},0)}= 
\begin{cases}
    B_{\beta},& \text{if } 1\leq r\leq 1+\beta, \\
   f_{(\phi,-\pi\rho(g))},              &  \text{if } r\leq 1,
\end{cases}
\end{equation}
where by Theorem \ref{ac}, $$B_{\beta}\circ \rho=\hat{B}_{\beta}:=Y_{\beta}({R}_{\beta},\theta)+\frac{{R}_{\beta}^{2}{\Theta}_{\beta}-{R}_{\beta}^{2}\theta}{2}.$$
By equality (\ref{imp}), we have the following 
\begin{align}\label{fixedbound}
\inf\bigg\{ A_{({\phi}_{\beta},0)}(x) | x\in P({\phi}_{\beta})\big|_{r\leq 1}\bigg\}=\inf\bigg\{ A_{({\phi},-\pi\rho(g))}(x) | x\in P(\phi)\bigg\}.
\end{align}
We now give lower and upper bounds for the action and  the Calabi invariant of the diffeomorphisms $\phi_{\beta}:\mathbb{D}(1+\beta)\to\mathbb{D}(1+\beta)$. 

\begin{claim}\label{Claim3} For every $\beta\in(0,1)$ and $r\in[1,1+\beta]$ the following the inequality
\begin{align}\label{bound2}
    |\hat{B}_{\beta}(r,\theta)|\leq 4\pi
\end{align} holds. Moreover,
\begin{align}\label{bound3}
  \bigg|\int_{[1,1+\beta]\times [0,2\pi]}\hat{B}_{\beta}rdr\wedge\theta\bigg| \leq 4\pi^{2}\beta(2+\beta).  
\end{align}
\end{claim}
\begin{claimproof}
Suppose $1>\beta>0$ and $1+\beta \geq r\geq 1.$  We know that $$|\hat{B}_{\beta}(r,\theta)|=$$
$$|Y_{\beta}({R}_{\beta},\theta)+\frac{D_{1}Y_{\beta}({R}_{\beta},\theta){R}_{\beta}}{2}|{\leq} |Y_{\beta}|+\frac{|D_{1}Y_{\beta}|(1+\beta)}{2}\leq$$
$$|Y_{\beta}-V_{\beta}|+|V_{\beta}|+\frac{|D_{1}Y_{\beta}-D_{1}V_{\beta}|(1+\beta)}{2}+\frac{|D_{1}V_{\beta}|(1+\beta)}{2}\overset{\ By \ Claim \ \ref{bound1}}{\leq}$$
$$|Y_{\beta}-V_{\beta}| + 2\beta\pi + \frac{|D_{1}Y_{\beta}-D_{1}V_{\beta}|(1+\beta)}{2}+\frac{2\pi(1+\beta)}{2} \overset{\ By \ property \ (1) \ of \ Y_{\beta} }{<}$$
$$\frac{6(1-\beta)\pi}{3+\beta}+2\beta\pi+\frac{3(1-\beta)(1+\beta)\pi}{3+\beta}+\pi(1+\beta)=4\pi.$$
We use the first part of claim to prove the second part:
$$\bigg|\int_{[1,1+\beta]\times [0,2\pi]}\hat{B}_{\beta}rdr\wedge\theta\bigg|\leq  \int_{[1,1+\beta]\times [0,2\pi]}|\hat{B}_{\beta}r|dr\wedge\theta\leq \ \ 4\pi  \int_{[1,1+\beta]\times [0,2\pi]}rdr\wedge\theta=4\pi^{2}\beta(2+\beta).$$\end{claimproof}

By inequality (\ref{bound2}) for every $\beta\in(0,1)$ and $r\in [1,1+\beta]$,  we have the following inequality 
\begin{align}\label{asympbound}
   |A^{\infty}_{({\phi}_{\beta},0)}(x)|\leq 4\pi.
\end{align}
Because 
$$|A^{\infty}_{({\phi}_{\beta},0)}(x)|\leq |\lim_{n\to\infty}\frac{\sum_{i=0}^{n}B({\phi}_{\beta}^{i}(x))}{n+1}|\leq \lim_{n\to\infty}\frac{\sum_{i=0}^{n}|B({\phi}_{\beta}^{i}(x))|}{n+1}\leq 4\pi.$$
Another useful inequality is the following:
\begin{align}\label{perbound}
   \inf\bigg\{ A^{\infty}_{({\phi}_{\beta},0)}(x)| x\in \{r\geq 1\}\bigg\}\leq \inf\bigg\{ A_{({\phi}_{\beta},0)}(x)| x\in P({\phi}_{\beta})\big|_{r\geq 1}\bigg\}. 
\end{align}
Applying estimate (\ref{bound3}), we get
\begin{align}\label{Calbound}
   \mathcal{V}{({\phi}_{\beta},0)}=\frac{1}{(1+\beta)^{2}\pi}\int_{[1,1+\beta]\times [0,2\pi]}f_{({\phi}_{\beta},0)}rdr\wedge\theta=
 \end{align}
$$\frac{1}{(1+\beta)^{2}}\mathcal{V}{(\phi,-\pi\rho(g))}+\frac{1}{(1+\beta)^{2}\pi}\int_{[1,1+\beta]\times [0,2\pi]}\hat{B}rdr\wedge\theta$$
$$\leq \frac{1}{(1+\beta)^{2}}\mathcal{V}{(\phi,-\pi\rho(g))}+\frac{4\pi\beta(2+\beta)}{(1+\beta)^{2}}<-4\pi $$
an estimate  for the Calabi invariant of the diffeomorphisms $\phi_{\beta}$ restricted to the disc $\mathbb{D}(1+\beta).$

By inequalities (\ref{Calbound}) and (\ref{asympbound}), we have
\begin{align}\label{asympbound1}
    \mathcal{V}{({\phi}_{\beta},0)}<-4\pi \leq \inf\bigg\{ A^{\infty}_{({\phi}_{\beta},0)}(x)| x\in \{r\geq 1\}\bigg\}.
\end{align}
By construction of $Y_{\beta},$ the diffeomorphism ${\phi}_{\beta}$ is identity near the boundary of $\mathbb{D}(1+\beta)$ and $\mathcal{V}({\phi}_{\beta}, 0)<-4\pi<0.$  Hence, by Theorem \ref{hut} \begin{align}\label{hut1}
    \inf\bigg\{ A_{({\phi}_{\beta},0)}(x)| x\in P({\phi}_{\beta})\bigg\}{\leq} \mathcal{V}({\phi}_{\beta}, 0).
\end{align}
Summarizing inequalities above, we get
$$-4\pi>\frac{1}{(1+\beta)^{2}}\mathcal{V}{(\phi,-\pi\rho(g))}+\frac{4\pi\beta(2+\beta)}{(1+\beta)^{2}}\overset{ \  By \  (\ref{Calbound})}{\geq}$$
$$\mathcal{V}({\phi}_{\beta}, 0)\overset{ \  By  \   (\ref{hut1})}{\geq}\inf\bigg\{ A_{({\phi}_{\beta},0)}(x)| x\in P({\phi}_{\beta})\bigg\}  =$$ 
$$ \min\bigg\{\inf\Big\{ A_{({\phi}_{\beta},0)}(x)| x\in P({\phi}_{\beta})\big|_{r\leq 1}\Big\}, \inf\Big\{ A_{({\phi}_{\beta},0)}(x)| x\in P({\phi}_{\beta})\big|_{r\geq 1}\Big\}\bigg\}\overset{ \   \ By \ (\ref{perbound})}{\geq}$$
$$\min\bigg\{\inf\Big\{ A_{({\phi}_{\beta},0)}(x)| x\in P({\phi}_{\beta})\big|_{r\leq 1}\Big\}, \inf\Big\{ A^{\infty}_{({\phi}_{\beta},0)}(x)| x\in \{r\geq 1\}\Big\}\bigg\}\overset{ \  By \  (\ref{asympbound1})}{=}$$
$$\inf\Big\{ A_{({\phi},-\pi\rho(g))}(x) | x\in P(\phi)\Big\}\overset{ \  By \  (\ref{fixedbound})}{=}\inf\Big\{ A_{({\phi}_{\beta},0)}(x)| x\in P({\phi}_{\beta})\big|_{r\leq 1}\Big\}$$
Taking the limit $\beta\to 0,$ we obtain
$$\inf\Big\{ A_{(\phi,-\pi\rho(g)))}(x)| x\in P(\phi)\Big\}\leq \mathcal{V}(\phi, -\pi\rho(g))).$$
Equivalently,
$$\inf\Big\{ A_{(\phi,0)}(x)| x\in P(\phi)\Big\}\leq \mathcal{V}(\phi, 0).$$\end{proof}




\bibliographystyle{amsalpha}
\bibliography{main}

\end{document}